\documentclass[12pt]{amsart}

\usepackage{amssymb}
\usepackage{verbatim}
\usepackage[toc,page]{appendix}
\usepackage{mathrsfs}

\newtheorem{thm}{Theorem}[section]

\newtheorem{lem}[thm]{Lemma}
\newtheorem{cor}[thm]{Corollary}

\theoremstyle{definition}

\theoremstyle{remark}

\newtheorem*{rem}{Remark}

\numberwithin{equation}{section}

\newcommand{\M}{\mathcal{M}}

\newcommand{\Mod}[1]{\ (\textup{mod}\ #1)}

\providecommand{\sym}{\operatorname{sym}}
\providecommand{\arcosh}{\operatorname{arcosh}}

\DeclareMathOperator{\res}{res}

\begin{document}

\title[]{The first moment of Maa{\ss} form symmetric square $L$-functions}

\begin{abstract}
We prove an asymptotic formula for the twisted first moment of Maa{\ss} form symmetric square $L$-functions on the critical line and at the critical point. The error term is estimated uniformly with respect to all parameters.
\end{abstract}

\author{Olga  Balkanova}
\address{Steklov Mathematical Institute of Russian Academy of Sciences, 8 Gubkina st., Moscow, 119991, Russia}
\email{olgabalkanova@gmail.com}

\maketitle

\tableofcontents


\section{Introduction}

Let $\{u_j\}$ be
the orthonormal basis of the space of Maa{\ss} cusp forms consisting of common eigenfunctions of all Hecke operators and the hyperbolic Laplacian. We denote by
$\{\lambda_{j}(n)\}$ the eigenvalues of Hecke operators acting on $u_{j}$ and by $\kappa_{j}=1/4+t_{j}^2$  the eigenvalues of the hyperbolic Laplacian acting on $u_{j}$.
Elements of the basis have a Fourier expansion of the following form
\begin{equation*}
u_{j}(x+iy)=\sqrt{y}\sum_{n\neq 0}\rho_{j}(n)K_{it_j}(2\pi|n|y)e(nx),
\end{equation*}
where $K_{\alpha}(x)$ is the $K$-Bessel function and
$\rho_{j}(n)=\rho_{j}(1)\lambda_{j}(n).$

The symmetric square $L$-function is defined by
\begin{equation}
L(\sym^2 u_{j},s):=\zeta(2s)\sum_{n=1}^{\infty}\frac{\lambda_{j}(n^2)}{n^s}
\end{equation}
for $\Re{s}>1$ and admits the analytic continuation to the whole complex plane.

For $T^{\epsilon}<G<T^{1-\epsilon}$  we introduce the test function 
\begin{equation}\label{omega def}
\omega_T(r):=\frac{1}{G\pi^{1/2}}\int_T^{2T}\exp\left(-\frac{(r-K)^2}{G^2}\right)dK
\end{equation}
and investigate the asymptotic behaviour of the twisted first moment
\begin{equation}
\M_1(l,s):=\sum_{j}\omega_T(t_j)\alpha_{j}\lambda_{j}(l^2)L(\sym^2 u_{j},s),
\end{equation}
on the critical line $s=1/2+it$ and at the critical point $s=1/2,$
where  
\begin{equation}\label{alphaj}
\alpha_{j}:=\frac{|\rho_{j}(1)|^2}{\cosh{\pi t_j}}
\end{equation}
is the normalizing coefficient.

\begin{thm}\label{thm rho=1/2+it} Assume that $T^{\epsilon}<G<T^{1-\epsilon}.$
For $s=1/2+it,$ $|t|\ll T^{1-\epsilon}/G$ and $l\ll G^2/T^{\epsilon}$ we have
\begin{equation}\label{M1lrho=1/2+it}
M_1(l,s)=MT(T;G;l;t)+O(T(1+|t|)^{\theta_t}l^{1/2+2\theta}+
T(1+|t|)^{\vartheta}l^{\epsilon}\log^3T),
\end{equation}
where $\theta=\theta_t=1/6+\epsilon$, $\vartheta=13/84+\epsilon$ and $MT(T;G;l;t)$ is the main term given by \eqref{eq:main term with t}.
\end{thm}

\begin{thm}\label{thm rho=1/2}
Assume that $T^{\epsilon}<G<T^{1-\epsilon}$. The following asymptotic formula holds
\begin{multline}\label{M1lrho=1/2}
M_1(l,1/2)=\frac{1}{\pi^2l^{1/2}}\int_{-\infty}^{\infty}r\omega_T(r)\tanh(\pi r)\Biggl(
\frac{3}{2}\gamma-\frac{\pi}{4}-\log l-\frac{3}{2}\log(2\pi)+\\+
\frac{\psi(1/4+ir)+\psi(1/4-ir)}{2}\Biggr)dr+O(Tl^{1/2+2\theta}),
\end{multline}
where $\theta=1/6+\epsilon$.
\end{thm}

Theorem \ref{thm rho=1/2} improves the results of Ng \cite[Theorem 7.1.1]{Ng} and Tang \cite{T}. 

The proof of Theorems \ref{thm rho=1/2+it} and \ref{thm rho=1/2} is based on the method of analytic continuation. More precisely, after gathering the required theoretical material in Section \ref{sec:prelim}, we derive an explicit formula for the twisted first moment in Section \ref{sec:explicit formula}. This allows us to obtain asymptotic formulas \eqref{M1lrho=1/2+it} and \eqref{M1lrho=1/2} by analyzing corresponding special functions in Section \ref{sec: main thms}.

\section{Preliminary results}\label{sec:prelim}

Let $\Gamma(z)$ be the Gamma function. By Stirling's formula we have
\begin{multline}\label{Stirling2}
\Gamma(\sigma+it)=\sqrt{2\pi}|t|^{\sigma-1/2}\exp(-\pi|t|/2)\\\times
\exp\left(i\left(t\log|t|-t+\frac{\pi t(\sigma-1/2)}{2|t|}\right)\right)
\left(1+O(|t|^{-1})\right)
\end{multline}
for $|t|\rightarrow\infty$ and a  fixed $\sigma$. 
We remark that instead of $O(|t|^{-1})$ it is possible to write arbitrarily accurate approximations by evaluating sufficiently many terms in the asymptotic expansion. 

Let $\vartheta$ be a subconvexity exponent for the Riemann zeta-function, namely $$|\zeta(1/2+ir)|\ll r^{\vartheta}.$$
 In this direction, the best known result $\vartheta=13/84+\epsilon$ is due to Bourgain \cite{Bo}.

For a complex number $ v $, let
\begin{equation*}
\tau_v(n):=\sum_{n_1n_2=n}\left( \frac{n_1}{n_2}\right)^v=n^{-v}\sigma_{2v}(n),
\end{equation*}
where
\begin{equation*}\sigma_v(n):=\sum_{d|n}d^v.
\end{equation*}

The following identity holds
\begin{equation}\label{series with tau}
\sum_{n=1}^{\infty}\frac{\tau_{ir}(n^2)}{n^s}=\frac{\zeta(s)\zeta(s+2ir)\zeta(s-2ir)}{\zeta(2s)}.
\end{equation}

 For $m,n\geq 1$ the Fourier cofficients satisfy the Hecke identity 
\begin{equation}\label{eq:multipFourcoeff2}
\lambda_{j}(n)\lambda_{j}(m)=\sum_{d|(m,n)}\lambda_{j}\left( \frac{nm}{d^2}\right).
\end{equation}

The following notation will be used for the Mellin transform
\begin{equation*}
\hat{f}(s):=\int_0^{\infty}f(x)x^{s-1}dx.
\end{equation*}

 The  Kloosterman sum is given by
\begin{equation*}
S(n,m;c):=\sum_{\substack{a\pmod{c}\\ (a,c)=1}}e\left( \frac{an+a^*m}{c}\right), \quad aa^*\equiv 1\pmod{c},
\end{equation*}
where $e(x):=exp(2\pi ix)$.
According to the Weil bound
\begin{equation}\label{Weilbound}
|S(m,n;c)|\leq \tau_0(c)\sqrt{(m,n,c)}\sqrt{c}.
\end{equation}

For $\Re{s}>1$ let
\begin{equation}\label{Lbyk}
\mathscr{L}_{n}(s):=\frac{\zeta(2s)}{\zeta(s)}\sum_{q=1}^{\infty}\frac{\rho_q(n)}{q^{s}}=\sum_{q=1}^{\infty}\frac{\lambda_q(n)}{q^{s}},
\end{equation}
where
\begin{equation}
\rho_q(n):=\#\{x\Mod{2q}:x^2\equiv n\Mod{4q}\},
\end{equation}
\begin{equation}
\lambda_q(n):=\sum_{q_{1}^{2}q_2q_3=q}\mu(q_2)\rho_{q_3}(n).
\end{equation}
Zagier \cite{Z} proved that $\mathscr{L}_{n}(s)$ admits meromorphic continuation to the entire complex plane.
Note that
\begin{equation}\label{Lbyk n=0}
\mathscr{L}_{0}(s)=\zeta(2s-1).
\end{equation}
For $n\neq 0$ the following subconvexity bound holds
\begin{equation}\label{eq:subconvexity}
\mathscr{L}_n(1/2+it)\ll n^{\theta}(1+|t|)^{\theta_t},
\end{equation}
where $\theta$ and $\theta_t$ are subconvexity exponents for Dirichlet $L$-functions of real primitive characters. We can take $\theta=\theta_t=1/6+\epsilon$ according to the results of Conrey $\&$ Iwaniec \cite{CI} and Young \cite{Y}.

Throughout the paper we will use different test functions $h(t)$ satisfying some of the following conditions:
\begin{description}
  \item[C1] $h(t)$ is an even function;
  \item[C2] $h(t)$ holomorphic in the strip $|\Im(t)|<\Delta$ for some $\Delta>1/2$;
  \item[C3] $h(t)$ is such that the following estimate $h(t)\ll(1+|t|)^{-2-\epsilon}$ holds in the strip $|\Im(t)|<\Delta$, $\Delta>1/2$;
  \item[C4] $h(\pm(n+1/2)i)=0$ for $n=0,1,\ldots N-1$, where $N>0$ is a sufficiently large integer.
\end{description}
Let
\begin{equation}\label{H0def}
H_0:=\int_{-\infty}^{\infty}rh(r)\tanh(\pi r)dr
\end{equation}
and
\begin{equation}\label{phidef}
\phi(x):=\frac{2i}{\pi}\int_{-\infty}^{\infty}J_{2ir}(x)h(r)\frac{rdr}{\cosh(\pi r)}.
\end{equation}

\begin{lem}(Kuznetsov trace formula)
For all $m,n \geq 1$ and any function $h(t)$  satisfying the conditions $(C1)-(C3)$ the following formula holds
\begin{multline}\label{eq:KuzTrForm}
\sum_{j=1}^{\infty}\alpha_j\lambda_j(m)\lambda_j(n)h(t_j)+
\frac{1}{\pi}\int_{-\infty}^{\infty}
\frac{\tau_{ir}(m)\tau_{ir}(n)}{|\zeta(1+2ir)|^2}h(r)dr=\\
\frac{\delta(m,n)}{\pi^2}H_0+
\sum_{q=1}^{\infty}\frac{S(m,n;q)}{q}\phi\left( \frac{4\pi \sqrt{mn}}{q}\right).
\end{multline}
\end{lem}
\begin{proof}
See \cite{Iwbook} or \cite{Kuz}.
\end{proof}

\begin{lem}\label{phi properties}
For any function $h(t)$  satisfying the conditions $(C1)-(C4)$ for $0<\Re{s}<3/2$ we have
\begin{equation}\label{phi Mellin}
\hat{\phi}(s)=\frac{2^{s}i}{\pi}\int_{-\infty}^{\infty}\frac{rh(r)}{\cosh(\pi r)}\frac{\Gamma(s/2+ir)}{\Gamma(1-s/2+ir)}dr.
\end{equation}
For $-1-2N<\Re{s}<3/2$ the following estimate holds
\begin{equation}\label{phi Mellin est}
\hat{\phi}(s)\ll (1+|\Im{s}|)^{\Re{s}-1}.
\end{equation}
\end{lem}
\begin{proof}
Equation \eqref{phi Mellin} follows immediately from \eqref{phidef} and \cite[(1), p. 326]{BE}. Using $(C4)$ we can move the line of integration in \eqref{phi Mellin}  to $\Im{r}=-c$ with $0<c<N+1/2$ without crossing any poles. Making the change of variables $z:=ir$ we obtain
\begin{equation}\label{phi Mellin2}
\hat{\phi}(s)=\frac{2^{s}}{\pi i}\int_{(c)}\frac{zh(iz)}{\cos(\pi z)}\frac{\Gamma(s/2+z)}{\Gamma(1-s/2+z)}dz
\end{equation}
for $-1-2N<\Re{s}<3/2.$ Applying  Stirling's formula \eqref{Stirling2} to estimate \eqref{phi Mellin2} we prove
\eqref{phi Mellin est}.
\end{proof}

\section{Explicit formula}\label{sec:explicit formula}
In this section we prove an explicit formula relating the first moment
\begin{equation}
\M_1(l,\rho;h):=\sum_{j}h(t_j)\alpha_{j}\lambda_{j}(l^2)L(\sym^2 u_{j},\rho)
\end{equation}
and sums of $\mathscr{L}_{n^2-4l^2}(\rho)$ weighted by the following integral
\begin{equation}\label{eq:integralI}
I(\rho;x):=\frac{2}{2\pi i}\int_{(\Delta)}\hat{\phi}(w)\Gamma(1-\rho-w)\sin\left( \pi \frac{\rho+w}{2}\right)(x/2)^wdw,
\end{equation}
where $-1-2N<\Delta<1-\Re{\rho}.$
\begin{thm}\label{thm rho<1 exact} For $0<\Re{\rho}<1$ and any function $h(t)$  satisfying the  conditions $(C1)-(C4)$ we have
\begin{multline}\label{eq:M1lrho<1}
M_1(l,\rho;h)=\frac{\zeta(2\rho)}{\pi^2l^{\rho}}\int_{-\infty}^{\infty}rh(r)\tanh(\pi r)dr+
(2\pi)^{\rho-1}\zeta(2\rho-1)I(\rho;2)-\\-
\frac{\zeta(\rho)}{\pi}\int_{-\infty}^{\infty}\frac{\tau_{ir}(l^2)\zeta(\rho+2ir)\zeta(\rho-2ir)}{|\zeta(1+2ir)|^2}h(r)dr-
\frac{2\zeta(2\rho-1)}{\zeta(2-\rho)}\tau_{(1-\rho)/2}(l^2)h\left(\frac{1-\rho}{2i}\right)
+\\
\frac{\hat{\phi}(1-\rho)}{(4\pi l)^{1-\rho}}\mathscr{L}_{-4l^2}(\rho)+
(2\pi)^{\rho-1}\sum_{n=1}^{2l-1}\frac{1}{n^{1-\rho}}\mathscr{L}_{n^2-4l^2}(\rho)I\left(\rho; \frac{n}{l}\right)+
(2\pi)^{\rho-1}\sum_{n=2l+1}^{\infty}\frac{1}{n^{1-\rho}}\mathscr{L}_{n^2-4l^2}(\rho)I\left(\rho; \frac{n}{l}\right).
\end{multline}
\end{thm}
The proof of Theorem \ref{thm rho<1 exact}  is based on the method of analytic continuation. 
Therefore, we first assume that $\Re{\rho}$ is sufficiently large.
\begin{lem} For $\Re{\rho}>3/2$ and any function $h(t)$  satisfying the conditions $(C1)-(C4)$ the following explicit formula holds
\begin{multline}\label{eq:M1lrho}
M_1(l,\rho;h)=\frac{\zeta(2\rho)}{\pi^2l^{\rho}}H_0-
\frac{\zeta(\rho)}{\pi}\int_{-\infty}^{\infty}\frac{\tau_{ir}(l^2)\zeta(\rho+2ir)\zeta(\rho-2ir)}{|\zeta(1+2ir)|^2}h(r)dr+\\
\frac{\hat{\phi}(1-\rho)}{(4\pi l)^{1-\rho}}\mathscr{L}_{-4l^2}(\rho)
+
(2\pi)^{\rho-1}\sum_{n=1}^{\infty}\frac{1}{n^{1-\rho}}\mathscr{L}_{n^2-4l^2}(\rho)I\left(\rho; \frac{n}{l}\right).
\end{multline}
\end{lem}
\begin{proof}
This is an analogue of \cite[Lemma 5.1]{BF1}. The only difference is that instead of the Petersson trace formula we now apply the Kuznetsov trace formula \eqref{eq:KuzTrForm}. Consequently, we have $\phi(4\pi ln/q)$ as the weight function in the sums of Kloosterman sums
instead of the Bessel function $J_{2k-1}(4\pi ln/q)$. Accordingly, it is required to replace everywhere in \cite[Lemma 5.1]{BF1} the Mellin transform of $J_{2k-1}(x)$ which is equal to
\begin{equation*}
\hat{J}_{2k-1}(w)=2^{w-1}\frac{\Gamma(k-1/2+w/2)}{\Gamma(k+1/2-w/2)}
\end{equation*}
by $\hat{\phi}(w).$ Finally, we remark that Lemma \ref{phi properties} guarantees that all sums and integrals remain to be absolutely convergent.
\end{proof}

To prove the analytic continuation of \eqref{eq:M1lrho} to the region $0<\Re{\rho}<1$ it is required to investigate
\begin{multline*}
C(l;\rho):=\frac{\zeta(\rho)}{\pi}\int_{-\infty}^{\infty}\frac{\tau_{ir}(l^2)\zeta(\rho+2ir)\zeta(\rho-2ir)}{|\zeta(1+2ir)|^2}
h(r)dr=\\
\frac{\zeta(\rho)}{2\pi i}
\int_{(0)}\frac{\tau_{z/2}(l^2)\zeta(\rho+z)\zeta(\rho-z)}{\zeta(1+z)\zeta(1-z)}h\left(\frac{z}{2i}\right)dz.
\end{multline*}
To this end, we apply \cite[Corollary 2.4.2, p. 55]{CMR}, which yields for $0<\Re{\rho}<1$ that
\begin{multline*}
C(l;\rho)=
\frac{\zeta(\rho)}{2\pi}\int_{(0)}\frac{\tau_{z/2}(l^2)\zeta(\rho+z)\zeta(\rho-z)}{\zeta(1+z)\zeta(1-z)}h\left(\frac{z}{2i}\right)dz+\\
\zeta(\rho)\left(\res_{z=1-\rho}-\res_{z=\rho-1} \right)\frac{\tau_{z/2}(l^2)\zeta(\rho+z)\zeta(\rho-z)}{\zeta(1+z)\zeta(1-z)}h\left(\frac{z}{2i}\right).
\end{multline*}
Evaluating the residues, applying \eqref{Lbyk n=0} and \eqref{H0def}, we prove Theorem \ref{thm rho<1 exact}.


We now proceed to investigate more closely the integral $I(\rho;x)$.
\begin{lem}
Assume that $x \geq 2$ and $0<\Re{\rho}<1$. Then
\begin{multline}\label{integralIgeq2}
I(\rho;x)=\frac{2^{2-\rho}i}{\pi^{3/2}}\int_{-\infty}^{\infty}\frac{rh(r)}{\cosh(\pi r)}\left(\frac{2}{x}\right)^{2ir}
\sin\left( \pi(\rho/2-ir)\right)\frac{\Gamma(1/2-\rho/2+ir)\Gamma(1-\rho/2+ir)}{\Gamma(1+2ir)}\times \\ {}_2F_{1}\left(1/2-\rho/2+ir,1-\rho/2+ir,1+2ir;\frac{4}{x^2} \right)dr.
\end{multline}
\end{lem}
\begin{proof}
Substituting \eqref{phi Mellin} to \eqref{eq:integralI} we obtain
\begin{multline}\label{Igeq2 1}
I(\rho;x)=\frac{2i}{\pi}\int_{-\infty}^{\infty}\frac{rh(r)}{\cosh(\pi r)}
\frac{1}{2\pi i}\int_{(\Delta)}\frac{\Gamma(w/2+ir)}{\Gamma(1-w/2+ir)}\times \\
\Gamma(1-\rho-w)\sin\left( \pi \frac{\rho+w}{2}\right)x^wdwdr
\end{multline}
with $0<\Delta<1-\Re{\rho}.$ Moving the line of integration to the left and evaluating the residues at the points $w=-2ir-2j,\, j=0,1,2\ldots$, we obtain
\begin{multline*}
I(\rho;x)=\frac{2i}{\pi}\int_{-\infty}^{\infty}\frac{rh(r)}{\cosh(\pi r)}
2\sin\left( \pi(\rho/2-ir)\right)x^{-2ir}\sum_{j=0}^{\infty}\frac{\Gamma(1-\rho+2ir+2j)x^{-2j}}{\Gamma(1+2ir+j)j!}
dr.
\end{multline*}
Applying \cite[Eq. 5.5.5]{HMF} we complete the proof of \eqref{integralIgeq2}.
\end{proof}

\begin{lem}
For $0<\Re{\rho}<1$ the following integral representation holds
\begin{equation}\label{integralIeq2}
I(\rho;2)=\Gamma(\rho-1/2)
\frac{2^{2-\rho}i}{\pi^{3/2}}\int_{-\infty}^{\infty}\frac{rh(r)}{\cosh(\pi r)}
\sin\left( \pi(\rho/2-ir)\right)\frac{\Gamma(1/2-\rho/2+ir)\Gamma(1-\rho/2+ir)}{\Gamma(\rho/2+ir)\Gamma(1/2+\rho/2+ir)}dr.
\end{equation}
\end{lem}
\begin{proof}
Letting $x=2$ in \eqref{integralIgeq2} and applying  \cite[Eq.~15.4.20]{HMF}, we find that
\begin{equation*}
{}_2F_{1}\left(1/2-\rho/2+ir,1-\rho/2+ir,1+2ir;1 \right)=
\frac{\Gamma(\rho-1/2)\Gamma(1+2ir)}{\Gamma(\rho/2+ir)\Gamma(1/2+\rho/2+ir)}.
\end{equation*}
The assertion follows.
\end{proof}

\begin{lem}
For $x<2$ and $0<\Re{\rho}<1$ we have
\begin{multline}\label{integralIleq2}
I(\rho;x)=\frac{2i}{\pi^{3/2}}\int_{-\infty}^{\infty}\frac{rh(r)}{\cosh(\pi r)}x^{1-\rho}
\cos\left( \pi(\rho/2+ir)\right)\frac{\Gamma(1/2-\rho/2+ir)\Gamma(1/2-\rho/2-ir)}{\Gamma(1/2)}\times \\ {}_2F_{1}\left(1/2-\rho/2+ir,1/2-\rho/2-ir,1/2;\frac{x^2}{4} \right)dr.
\end{multline}
\end{lem}
\begin{proof}
Moving the contour of integration in \eqref{Igeq2 1} to the right we cross the simple poles at $w=1-\rho+j$, $j=0,1,2,\ldots$ Consequently,
\begin{equation*}
I(\rho;x)=\frac{2i}{\pi}\int_{-\infty}^{\infty}\frac{rh(r)}{\cosh(\pi r)}x^{1-\rho}
\sum_{m=0}^{\infty}\frac{(-1)^m\Gamma(1/2-\rho+ir+m)x^{2m}}{\Gamma(1/2+\rho+ir-m)(2m)!}
dr.
\end{equation*}
Applying \cite[Eq. 5.5.3]{HMF} for $\Gamma(1/2+\rho+ir-m)$ and \cite[Eq. 5.5.5]{HMF} for $\Gamma(2m+1)$, we prove \eqref{integralIleq2}.
\end{proof}

\begin{lem}
For $0<z<2$ and $0<\Re{\rho}<1$ we have
\begin{equation}\label{Ileq2}
I(\rho;z)=z^{1-\rho}\frac{2^{1+\rho}i}{\pi}\int_{-\infty}^{\infty}\frac{rh(r)}{\cosh(\pi r)}
\int_0^{\infty}x^{-\rho}J_{2ir}(x)
\cos\left(\frac{xz}{2}\right)dxdr.
\end{equation}
\end{lem}
\begin{proof}
Applying  \cite[Eq. 6.699.2]{GR} we have for $0<z<2,$ $-1/2<\Re{\rho}<1$
\begin{multline*}
\int_0^{\infty}x^{-\rho}J_{2ir}(x)\cos\left(\frac{xz}{2}\right)dx=
\frac{2^{-\rho}\Gamma(1/2-\rho/2+ir)}{\Gamma(1/2+\rho/2+ir)}\times \\
{}_2F_{1}\left(1/2-\rho/2+ir,1/2-\rho/2-ir,1/2;\frac{z^2}{4} \right).
\end{multline*}
Using \cite[Eq. 5.5.3]{HMF} we infer that
\begin{multline}\label{bessel to F21}
\int_0^{\infty}x^{-\rho}J_{2ir}(x)\cos\left(\frac{xz}{2}\right)dx=
\frac{\cos\left( \pi(\rho/2+ir)\right)}{2^{\rho}\pi}\Gamma(1/2-\rho/2+ir)\Gamma(1/2-\rho/2-ir)\times \\ {}_2F_{1}\left(1/2-\rho/2+ir,1/2-\rho/2-ir,1/2;\frac{x^2}{4} \right).
\end{multline}
Finally, \eqref{Ileq2} follows from \eqref{integralIleq2} and \eqref{bessel to F21}.
\end{proof}

\begin{rem}
The formula \eqref{Ileq2} is also valid for $z\geq 2.$ It is possible to rewrite \eqref{Ileq2} as follows
\begin{equation*}
I(\rho;z)=z^{1-\rho}2^{\rho}
\int_0^{\infty}\phi(x)x^{-\rho}\cos\left(\frac{xz}{2}\right)dx,
\end{equation*}
where $\phi(x)$ is given by \eqref{phidef}.
\end{rem}


As the final step, we derive an explicit formula at the critical point. 
\begin{thm}\label{thm rho=1/2 exact} For  any function $h(t)$  satisfying the conditions $(C1)-(C4)$ we have
\begin{multline}\label{eq:M1lrho=1/2}
M_1(l,1/2;h)=\frac{1}{\pi^2l^{1/2}}\int_{-\infty}^{\infty}rh(r)\tanh(\pi r)\Biggl(
\frac{3}{2}\gamma-\frac{\pi}{4}-\log l-\frac{3}{2}\log(2\pi)+\\+
\frac{\psi(1/4+ir)+\psi(1/4-ir)}{2}\Biggr)dr-
\frac{\zeta(1/2)}{\pi}\int_{-\infty}^{\infty}\frac{\tau_{ir}(l^2)\zeta(1/2+2ir)\zeta(1/2-2ir)}{|\zeta(1+2ir)|^2}h(r)dr-\\-
\frac{2\zeta(0)}{\zeta(3/2)}\tau_{1/4}(l^2)h\left(\frac{1}{4i}\right)
+\frac{\hat{\phi}(1/2)}{(4\pi l)^{1/2}}\mathscr{L}_{-4l^2}(1/2)+\\
(2\pi)^{-1/2}\sum_{n=1}^{2l-1}\frac{1}{n^{1/2}}\mathscr{L}_{n^2-4l^2}(1/2)I\left(\frac{1}{2}; \frac{n}{l}\right)+
(2\pi)^{-1/2}\sum_{n=2l+1}^{\infty}\frac{1}{n^{1/2}}\mathscr{L}_{n^2-4l^2}(\rho)I\left(\frac{1}{2}; \frac{n}{l}\right).
\end{multline}
\end{thm}
\begin{proof}
In order to prove \eqref{eq:M1lrho=1/2} it is required (see \eqref{eq:M1lrho<1}) to investigate
\begin{equation*}
\lim_{\rho\rightarrow1/2}\left(\frac{\zeta(2\rho)}{\pi^2l^{\rho}}\int_{-\infty}^{\infty}rh(r)\tanh(\pi r)dr+
(2\pi)^{\rho-1}\zeta(2\rho-1)I(\rho;2)\right).
\end{equation*}
Let $\rho=1/2+u$. Applying \eqref{integralIeq2}, we reduce our problem to evaluting
\begin{multline}\label{lim 1}
\lim_{u\rightarrow0}\Biggl(
\frac{\zeta(1+2u)}{\pi^2l^{1/2+u}}\int_{-\infty}^{\infty}rh(r)\tanh(\pi r)dr+
\zeta(2u)\Gamma(u)
\frac{2i}{\pi^{2-u}(2l)^{1/2-u}}\times \\
\int_{-\infty}^{\infty}\frac{rh(r)}{\cosh(\pi r)}
\sin\left( \pi(1/4+u/2-ir)\right)\frac{\Gamma(1/4-u/2+ir)\Gamma(3/4-u/2+ir)}{\Gamma(1/4+u/2+ir)\Gamma(3/4+u/2+ir)}dr
\Biggr).
\end{multline}
Using the functional equation for the Riemann zeta-function \cite[Eq. 25.4.1]{HMF}, namely
\begin{equation}\label{zeta}
\zeta(2u)\Gamma(u)=(2\pi)^{2u}\frac{\Gamma(1-2u)}{\Gamma(1-u)}\zeta(1-2u),
\end{equation}
we obtain that the limit \eqref{lim 1} is equal to
\begin{multline}\label{lim 2}
\int_{-\infty}^{\infty}rh(r)\Biggl(
\frac{2\gamma\tanh(\pi r)}{\pi^2l^{1/2}}-
\frac{\log l}{2\pi^2l^{1/2}}\tanh(\pi r)-
\frac{i}{\pi^2(2l)^{1/2}}\frac{\sin(\pi(1/4-ir))}{\cosh(\pi r)}\Bigl(\\
\log(8\pi^3l)-\psi(1)-\psi(1/4+ir)-\psi(3/4+ir)
\Bigr)-
\frac{i}{\pi^2(2l)^{1/2}}\frac{\pi\cos(\pi(1/4-ir))}{2\cosh(\pi r)}\Biggr)dr.
\end{multline}
According to  \cite[Eq. 5.5.4]{HMF} we have
\begin{equation}\label{psi=psi}
\psi(3/4+ir)=\psi(1/4-ir)+\frac{\pi}{\tan(1/4-ir)}.
\end{equation}
Substituting \eqref{psi=psi} to \eqref{lim 2}, applying the following identities
\begin{equation*}
\sin(\pi(1/4-ir))=\sin(\pi/4)\cosh(\pi r)-i\cos(\pi/4)\sinh(\pi r),
\end{equation*}
\begin{equation*}
\cos(\pi(1/4-ir))=\cos(\pi/4)\cosh(\pi r)+i\sin(\pi/4)\sinh(\pi r),
\end{equation*}
and  using the facts that $h(r)$ is even and that the integral of an odd function over the real line is equal to zero, we finally complete the proof.

\end{proof}

\section{Proof of main theorems}\label{sec: main thms}

First, let us recall some properties of the function $\omega_T(r)$ defined by \eqref{omega def}.
Assume that $T^{\epsilon}<G<T^{1-\epsilon}.$ For any $A>1$ and some $c>0$ we have (see \cite{IvJut})
\begin{equation}\label{omega1}
\omega_T(r)=1+O(r^{-A})\text{ if } T+cG\sqrt{\log T}<r<2T-cG\sqrt{\log T},
\end{equation}
\begin{equation}\label{omega2}
\omega_T(r)=O((|r|+T)^{-A})\text{ if }
r<T-cG\sqrt{\log T}\text{ or } r>2T+cG\sqrt{\log T},
\end{equation}
and otherwise
\begin{equation}\label{omega3}
\omega_T(r)=1+O(G^3(G+\min(|r-T|,|r-2T|))^{-3}).
\end{equation}
Furthermore, the following trivial estimates hold
\begin{equation}\label{omega4}
\omega_T(r)\ll TG^{-1}\exp\left(-\frac{(r-T)^2}{G^2}\right)\text{ if } r<T,
\end{equation}
\begin{equation}\label{omega5}
\omega_T(r)\ll TG^{-1}\exp\left(-\frac{(r-2T)^2}{G^2}\right)\text{ if } r>2T.
\end{equation}
For an arbitrary large integer $N$, we define
\begin{equation}\label{qN def}
q_N(r):=\frac{(r^2+1/4)\ldots(r^2+(N-1/2)^2)}{(r^2+100N^2)^N},
\end{equation}
\begin{equation}\label{hN def}
h(K,N;r):=q_N(r)\exp\left(-\frac{(r-K)^2}{G^2}\right)+q_N(r)\exp\left(-\frac{(r+K)^2}{G^2}\right).
\end{equation}
Note that $q_N(r)=1+O(1/r^2)$ and that the function $h(K,N;r)$  satisfies the conditions $(C1)-(C4).$ 




\subsection{Proof of Theorem \ref{thm rho=1/2+it}}
In order to prove Theorem \ref{thm rho=1/2+it}, we apply Theorem \ref{thm rho<1 exact} with the test function \eqref{hN def} and then integrate the result over $K$ to obtain \eqref{omega def}.

Taking $\rho=1/2+it$ in Theorem \ref{thm rho<1 exact}, we find that the main term is equal to 
\begin{multline}\label{eq:main term with t}
MT(T;G;l;t)=\frac{\zeta(2s)}{\pi^2l^{s}}\int_{-\infty}^{\infty}r(\omega_T(r)+\omega_T(-r))\tanh(\pi r)dr\\+
(2\pi)^{s-1}\zeta(2s-1)\Gamma(s-1/2)
\frac{2^{2-s}i}{\pi^{3/2}}\times\\
\int_{-\infty}^{\infty}\frac{r(\omega_T(r)+\omega_T(-r))}{\cosh(\pi r)}
\sin\left( \pi(s/2-ir)\right)\frac{\Gamma(1/2-s/2+ir)\Gamma(1-s/2+ir)}{\Gamma(s/2+ir)\Gamma(1/2+s/2+ir)}dr.
\end{multline}

Next, we proceed to estimate the contribution of the remaining terms. 
In the arguments below we will frequently use \cite[Eq. 3.323.2]{GR} and \cite[Eq. 3.462.2]{GR}, namely
\begin{equation}\label{exp integral1}
\int_{-\infty}^{\infty}\exp(-p^2x^2+qx)dx=\frac{\pi^{1/2}}{p}\exp\left(\frac{q^2}{4p^2}\right),
\quad\hbox{if}\quad\Re(p^2)>0,
\end{equation}
\begin{equation}\label{exp integral2}
\int_{-\infty}^{\infty}x^n\exp(-x^2+qx)dx=P_n(q)\exp\left(\frac{q^2}{4}\right),
\end{equation}
where $n$ is a nonnegative integer and $P_n(q)$ is a polynomial of degree $n.$


\begin{lem}
For $\rho=1/2+2it$, $|t|\ll T^{1-\epsilon},$ $|t|\ll T^{2-\epsilon}/G^2$ and $x\gg T^{\epsilon}$ the following estimate holds
\begin{equation}\label{Igeq2 estimate0}
\frac{1}{G\pi^{1/2}}\int_T^{2T}I(\rho;x)dK\ll x^{-2N}T^{2}.
\end{equation}
\end{lem}
\begin{proof}
Using the integral representation \cite[Eq. 15.6.1]{HMF} for the hypergeometric function in \eqref{integralIgeq2}, we obtain
\begin{multline}\label{Igeq2est1}
I(\rho;x)=\frac{2^{2-\rho}i}{\pi^{3/2}}
\int_0^1\frac{y^{-1/4-it}(1-y)^{-3/4+it}}{(1-4y/x^2)^{1/4-it}}
\int_{-\infty}^{\infty}\frac{rh(r)}{\cosh(\pi r)}
\sin\left( \pi(\rho/2-ir)\right)\times\\
\frac{\Gamma(1/4+i(r-t))}{\Gamma(1/4+i(r+t))}
\left(\frac{y(1-y)}{x^2/4-y}\right)^{ir}drdy.
\end{multline}
Since \eqref{hN def} contains the multiple $q_N(r)$ defined  by \eqref{qN def}, we do not cross any poles by moving the line of integration \eqref{Igeq2est1} to $\Im{r}=-N$. Applying the Stirling formula \eqref{Stirling2} we show that
\begin{equation*}
\frac{\sin\left( \pi(\rho/2-ir)\right)}{\cosh(\pi r)}\frac{\Gamma(1/4+i(r-t))}{\Gamma(1/4+i(r+t))}\ll
\frac{|r+t|^{1/4}}{|r-t|^{1/4}}\exp\left(\max(2t-2|r|,0)\right).
\end{equation*}
Since $|t|\ll T^{1-\epsilon}$ and $|t|\ll T^{2-\epsilon}/G^2$ we have
\begin{equation}\label{Gamma est1}
\exp\left(\max(2t-2|r|,0)\right)h(r)\ll 1.
\end{equation}
Using \eqref{Gamma est1} we estimate all integrals in \eqref{Igeq2est1} trivially  and  derive \eqref{Igeq2 estimate0}.
\end{proof}

\begin{lem}\label{lemma I est z>2}
For $\rho=1/2+2it$, $|t|\ll T^{1-\epsilon}/G$ and $x>2$  we have
\begin{equation}\label{Igeq2 estimate1}
\frac{1}{G\pi^{1/2}}\int_T^{2T}I(\rho;x)dK\ll
\exp\left(-G^2\log^2\left(x/2+\sqrt{x^2/4-1}\right)\right)+T^{-A},
\end{equation}
where $A>0$ is an arbitrary fixed number.
\end{lem}
\begin{proof}
It is sufficient to estimate the contribution of the first summand in \eqref{hN def}. The second summand can be treated similarly. Our strategy is to expand all multiples in \eqref{Igeq2est1} (except the exponent in \eqref{hN def}) in the Taylor series. 
We start by restricting the integral over $r$ in \eqref{Igeq2est1} to the range $|r-K|<G\log^2K$ at the cost of a negligible error term. Next, using the Stirling formula \eqref{Stirling2} and the fact that $|t|\ll T^{1-\epsilon}/G$  we obtain
\begin{multline*}
\frac{\sin\left( \pi(\rho/2-ir)\right)}{\cosh(\pi r)}\frac{\Gamma(1/4+i(r-t))}{\Gamma(1/4+i(r+t))}=
\frac{(r+t)^{1/4}}{(r-t)^{1/4}}\times\\
\exp\left(i\left(2t+(r-t)\log(r-t)-(r+t)\log(r+t)\right)\right)
\left(1+\ldots\right).
\end{multline*}
Note that
\begin{multline*}
2t+(r-t)\log(r-t)-(r+t)\log(r+t)=2t-2t\log r+\\
r\left((1-t/r)\log(1-t/r)-(1+t/r)\log(1+t/r)\right)=
-2t\log r+\frac{t^3}{3r^2}+\frac{t^5}{10r^4}+\ldots
\end{multline*}
Therefore, the $r$-integral in \eqref{Igeq2est1} is equal to (up to similar but smaller terms and a negligible error term)
\begin{multline*}
\int_{K-G\log^2 T}^{K+G\log^2 T}r\exp\left(-\frac{(r-K)^2}{G^2}\right)\exp\left(-i2t\log r
+\frac{it^3}{3r^2}+\frac{it^5}{10r^4}+\ldots\right)\times\\
\left(\frac{y(1-y)}{x^2/4-y}\right)^{ir}dr.
\end{multline*}
Making the change of variables $r:=K+Gx$, we can rewrite the main term as follows
\begin{multline}\label{rint 1}
G\int_{-\log^2 T}^{\log^2 T}\exp\left(-x^2\right)\exp\left(-i2t\log(K+Gx)
+\frac{it^3}{3(K+Gx)^2}+\frac{it^5}{10(K+Gx)^4}+\ldots\right)\times\\
\left(\frac{y(1-y)}{x^2/4-y}\right)^{i(K+Gx)}dx.
\end{multline}
Note that $|x|<\log^2 T$  and $|t|\ll T^{1-\epsilon}/G$ and
\begin{multline}\label{Taylor1}
\exp\left(-i2t\log(K+Gx)\right)=\exp\left(-i2t\log K\right)\exp\left(-i2t\frac{Gx}{K}+\ldots\right)=\\
\exp\left(-i2t\log K\right)\left(1-i2t\frac{Gx}{K}+\ldots\right).
\end{multline}
Consequently, the expression \eqref{rint 1} is equal to
\begin{equation*}
G\exp\left(-i2t\log K\right)\int_{-\log^2 T}^{\log^2 T}\exp\left(-x^2\right)
\left(\frac{y(1-y)}{x^2/4-y}\right)^{i(K+Gx)}dx
\end{equation*}
plus the same terms of smaller order and plus a negligible error term.  Furthermore, it is possible to extend the interval of integration to $(-\infty,\infty)$ at the cost of a negligible error term. Applying \eqref{exp integral1}, we obtain
\begin{equation*}
\pi^{1/2}G\exp\left(-i2t\log K\right)\left(\frac{y(1-y)}{x^2/4-y}\right)^{iK}
\exp\left(-\frac{G^2}{4}\log^2\left(\frac{y(1-y)}{x^2/4-y}\right)\right).
\end{equation*}
Substituting this to \eqref{Igeq2est1} we have
\begin{multline}\label{Igeq2est2}
I(\rho;x)=\frac{2^{2-\rho}i}{\pi^{3/2}}\pi^{1/2}G\exp\left(-i2t\log K\right)
\int_0^1\frac{y^{-1/4-it}(1-y)^{-3/4+it}}{(1-4y/x^2)^{1/4-it}}\times\\
\left(\frac{y(1-y)}{x^2/4-y}\right)^{iK}
\exp\left(-\frac{G^2}{4}\log^2\left(\frac{y(1-y)}{x^2/4-y}\right)\right)dy+\ldots
\end{multline}
Estimating the integral over $y$ trivially we infer
\begin{equation}\label{Igeq2est3}
\frac{1}{G\pi^{1/2}}\int_T^{2T}I(\rho;x)dK\ll
\max_{0<y<1}\exp\left(-\frac{G^2}{4}\log^2\left(\frac{y(1-y)}{x^2/4-y}\right)\right)+T^{-A}.
\end{equation}
The maximum is achieved at the point $y_0=\frac{x^2}{4}-\frac{x}{2}\left(\frac{x^2}{4}-1\right)^{1/2}$. Consequently,
\begin{equation}\label{Igeq2est4}
\frac{1}{G\pi^{1/2}}\int_T^{2T}I(\rho;x)dK\ll
\exp\left(-G^2\log^2\left(x/2+\sqrt{x^2/4-1}\right)\right)+T^{-A}.
\end{equation}
\end{proof}

\begin{cor}\label{cor I est z>2}
For $\rho=1/2+2it$, $|t|\ll T^{1-\epsilon}/G$ and $l\ll G^2T^{-\epsilon}$  we have
\begin{equation}\label{sumIgeq2 estimate1}
\sum_{n=2l+1}^{\infty}\frac{1}{n^{1-\rho}}\mathscr{L}_{n^2-4l^2}(\rho)
\frac{1}{G\pi^{1/2}}\int_T^{2T}I\left(\rho; \frac{n}{l}\right)dK\ll T^{-A},
\end{equation}
where $A>0$ is an arbitrary fixed number.
\end{cor}
\begin{proof}
Combining \eqref{Igeq2 estimate0}, \eqref{Igeq2 estimate1}  and \eqref{eq:subconvexity}, we show that the left-hand side of \eqref{sumIgeq2 estimate1} is bounded by
\begin{equation}\label{sumIgeq2 estimate2}
T^{-A}+
(1+|t|)^{\theta_t}\sum_{2l<n<T^{\epsilon}}\frac{(n^2-4l^2)^{\theta}}{n^{1/2}}
\exp\left(-G^2\log^2\left(n/(2l)+\sqrt{n^2/(4l^2)-1}\right)\right).
\end{equation}
For $n\geq2l+1$ we have
\begin{equation*}
G^2\log^2\left(n/(2l)+\sqrt{n^2/(4l^2)-1}\right)\gg \frac{G^2}{l}\gg T^{\epsilon},
\end{equation*}
and therefore, the sum in \eqref{sumIgeq2 estimate2} is negligibly small.
\end{proof}

\begin{lem}\label{lemma I est z<2}
For $0<z<2$ and $\rho=1/2+2it$ the following estimate holds
\begin{equation}\label{Ileq2 estimate1}
\frac{1}{G\pi^{1/2}}\int_T^{2T}I(\rho;z)dK\ll T\frac{z}{(2-z)^{1/2}}.
\end{equation}
\end{lem}
\begin{proof}
Consider the Bessel kernel
\begin{equation*}\label{eq:k+}
k^{+}(y,v)=\frac{1}{2\cos{\pi v}}\left( J_{2v-1}(y)-J_{1-2v}(y) \right).
\end{equation*}
According to \cite[Eq. (4.5)]{BF}
\begin{multline*}
k^{+}(a,1/2+it)=\frac{2}{\pi}\int_{0}^{\infty}\cos{(a\cosh{z})}\cos{(2tz)}dz
=\\
=\frac{2}{\pi}\int_{1}^{\infty}\frac{\cos{(au)}}{\sqrt{u^2-1}}\cos{\left( 2t \log{(u+\sqrt{u^2-1})}\right)}du.
\end{multline*}
Integration by parts gives (see \cite[Eq. (4.11)]{BF})
\begin{equation}\label{k+a}
k^{+}(a,1/2+it)=\frac{2}{\pi}\int_{0}^{U}\cos{(a\cosh{u})}\cos{(2tu)}du
+O\left( \frac{1+|t|}{a\exp(U)}\right).
\end{equation}
Since $h(r)$ is even, we can rewrite \eqref{Ileq2}  as
\begin{equation}\label{Ileq2 2}
I(\rho;z)=z^{1-\rho}\frac{2^{2+\rho}}{\pi}\int_{0}^{\infty}rh(r)\tanh(\pi r)
\int_0^{\infty}x^{-\rho}k^{+}(x,1/2+ir)\cos\left(\frac{xz}{2}\right)dxdr.
\end{equation}
Let $\epsilon>0$ be some small number (to be chosen later). Then
\begin{multline*}
\int_0^{\infty}x^{-\rho}k^{+}(x,1/2+ir)\cos\left(\frac{xz}{2}\right)dx=
\int_{\epsilon}^{\infty}x^{-\rho}k^{+}(x,1/2+ir)\cos\left(\frac{xz}{2}\right)dx+O(\epsilon^{1/2})=\\
\frac{2}{\pi}\int_{0}^{U}\cos{(2ru)}\int_{\epsilon}^{\infty}x^{-\rho}\cos{(x\cosh{u})}\cos\left(\frac{xz}{2}\right)dxdu+
O\left(\epsilon^{1/2}+ \frac{1+|r|}{\exp(U)\epsilon^{1/2}}\right)=\\
\frac{2}{\pi}\int_{0}^{U}\cos{(2ru)}\int_{0}^{\infty}x^{-\rho}\cos{(x\cosh{u})}\cos\left(\frac{xz}{2}\right)dxdu+
O\left(U\epsilon^{1/2}+ \frac{1+|r|}{\exp(U)\epsilon^{1/2}}\right).
\end{multline*}
To evaluate the integral over $x$ we apply \cite[Eq. (21), p.319]{BE}, getting
\begin{multline}\label{k+cos integral}
\int_0^{\infty}x^{-\rho}k^{+}(x,1/2+ir)\cos\left(\frac{xz}{2}\right)dx=
\frac{1}{\pi}\Gamma(1-\rho)\sin(\pi\rho/2)\times\\
\int_{0}^{U}\cos{(2ru)}\int_{0}^{\infty}\left(
(\cosh{u}+z/2)^{\rho-1}+(\cosh{u}-z/2)^{\rho-1}
\right)du+\\
O\left(U\epsilon^{1/2}+ \frac{1+|r|}{\exp(U)\epsilon^{1/2}}\right).
\end{multline}
Substituting \eqref{k+cos integral} to \eqref{Ileq2 2} and choosing $\epsilon=T^{-2A}$, $U=\log^2T$, we infer
\begin{multline}\label{Ileq2 3}
I(\rho;z)=z^{1-\rho}\frac{2^{2+\rho}}{\pi^2}\Gamma(1-\rho)\sin(\pi\rho/2)
\int_{0}^{U}\left(
(\cosh{u}+z/2)^{\rho-1}+(\cosh{u}-z/2)^{\rho-1}\right)\times\\
\int_{0}^{\infty}rh(r)\tanh(\pi r)
\cos{(2ru)}drdu+O(T^{-A}).
\end{multline}
Arguing in the same way as in the proof of Lemma \ref{lemma I est z>2}, we conclude that up to a negligible error term the integral over $r$ is equal to
\begin{equation*}\label{Ileq2 2}
KG\int_{-\infty}^{\infty}\exp(-x^2)\cos{(2u(K+Gx))}dx
\end{equation*}
plus integrals of the same form but of smaller size. Applying \cite[Eq. 3.896.2]{GR} to evaluate the integral, we have
\begin{equation}\label{rint2}
KG\int_{-\infty}^{\infty}\exp(-x^2)\cos{(2u(K+Gx))}dx=K\pi^{1/2}\exp(-u^2G^2)\cos{(2uK)}.
\end{equation}
Using \eqref{Ileq2 3} and \eqref{rint2}, we show that
\begin{multline*}
\frac{1}{G\pi^{1/2}}\int_T^{2T}I(\rho;z)dK\ll
z^{1/2}\int_{0}^{U}\left((\cosh{u}+z/2)^{-1/2}+(\cosh{u}-z/2)^{-1/2}\right)\times\\
\exp(-u^2G^2)\int_T^{2T}K\cos{(2uK)}dKdu.
\end{multline*}
Then the estimate
\begin{equation*}
\int_T^{2T}K\cos{(2uK)}dK\ll\min(T^2,T/u)
\end{equation*}
implies that
\begin{multline*}
\frac{1}{G\pi^{1/2}}\int_T^{2T}I(\rho;z)dK\ll
z^{1/2}T^2\int_{0}^{1/T}\left((\cosh{u}+z/2)^{-1/2}+(\cosh{u}-z/2)^{-1/2}\right)du+\\
z^{1/2}T\int_{1/T}^{U}\left((\cosh{u}+z/2)^{-1/2}+(\cosh{u}-z/2)^{-1/2}\right)
\exp(-u^2G^2)\frac{du}{u}\ll\\
T\frac{z}{(2-z)^{1/2}}.
\end{multline*}
\end{proof}
\begin{cor}\label{cor I est z<2}
For $\rho=1/2+2it$ we have
\begin{equation}\label{sumIleq2 estimate1}
\sum_{n=1}^{2l-1}\frac{1}{n^{1-\rho}}\mathscr{L}_{n^2-4l^2}(\rho)
\frac{1}{G\pi^{1/2}}\int_T^{2T}I\left(\rho; \frac{n}{l}\right)dK\ll
T^{1+\epsilon}(1+|t|)^{\theta_t}l^{1/2+2\theta}.
\end{equation}
\end{cor}
\begin{proof}
Using \eqref{Ileq2 estimate1}  and \eqref{eq:subconvexity}, we prove that the left-hand side of \eqref{sumIleq2 estimate1} is bounded by
\begin{multline*}
T(1+|t|)^{\theta_t}\sum_{n=1}^{2l-1}\frac{4l^2-n^2)^{\theta}}{l^{1/2}(2-n/l)^{1/2}}\ll
T(1+|t|)^{\theta_t}l^{\theta}\sum_{n=1}^{2l-1}(2l-n)^{\theta-1/2}\ll
T(1+|t|)^{\theta_t}l^{1/2+2\theta}.
\end{multline*}
\end{proof}

\begin{lem}\label{lemma phi hat}
For $\rho=1/2+2it$, $|t|\ll T^{1-\epsilon}$ we have
\begin{equation}\label{hat phi estimate}
\frac{1}{G\pi^{1/2}}\int_T^{2T}\hat{\phi}(1-\rho)dK\ll T^{-A},
\end{equation}
where $A$ is an arbitrary positive number.
\end{lem}
\begin{proof}
Applying the Stirling formula \eqref{Stirling2}, we estimate \eqref{phi Mellin}, getting
\begin{equation*}
\hat{\phi}(1-\rho)\ll\int_{-\infty}^{\infty}\frac{rh(r)}{(1+|r-t|)^{1/4}(1+|r+t|)^{1/4}}
\exp\left(-\frac{\pi}{2}\left(2|r|+|r-t|+|r+t|\right)\right)
dr.
\end{equation*}
Since $|t|\ll T^{1-\epsilon}$, we prove \eqref{hat phi estimate} by using\eqref{hN def}.
\end{proof}

\begin{lem}\label{lemma zeta}
For $\rho=1/2+2it$, $|t|\ll T^{1-\epsilon}$ we have
\begin{equation}\label{cont.part estimate}
\frac{1}{G\pi^{1/2}}\int_T^{2T}
\frac{\zeta(\rho)}{\pi}\int_{-\infty}^{\infty}\frac{\tau_{ir}(l^2)\zeta(\rho+2ir)\zeta(\rho-2ir)}{|\zeta(1+2ir)|^2}h(r)dr
dK\ll T(1+|t|)^{\vartheta}l^{\epsilon}\log^3T.
\end{equation}
\end{lem}
\begin{proof}
Using the standard estimate $|\zeta(1+2ir)|^{-1}\ll\log(1+|r|)$ and
\begin{equation*}
\int_R^{2R}|\zeta(1/2+ir)|^2dr\ll R\log R,
\end{equation*}
and applying the Cauchy-Schwarz inequality, we show that the left-hand side of \eqref{cont.part estimate} is bounded by
\begin{equation*}
l^{\epsilon}(T+|t|)\log(T+|t|)\log^2 T|\zeta(1/2+2it)|.
\end{equation*}
Finally, using the estimate $|\zeta(1/2+ir)|\ll r^{\vartheta}$,  we prove \eqref{cont.part estimate}.
\end{proof}

Finally, Theorem \ref{thm rho=1/2+it} follows from Corollary \ref{cor I est z>2}, Corollary \ref{cor I est z<2}, Lemma \ref{lemma phi hat} and Lemma \ref{lemma zeta}.


\subsection{Proof of Theorem \ref{thm rho=1/2}}
The only difference between Theorem \ref{thm rho=1/2} and Theorem \ref{thm rho=1/2+it} is in the range of $l.$ Indeed, if $l\gg G^2T^{-\epsilon}$, then the sum \eqref{sumIgeq2 estimate2}  is no longer negligible. To handle $I(1/2,z)$ for $z$ in the neighbourhood of $1$, we use the following modification of Lemma $2.4$ in \cite{Zav} that was proved in \cite[Corollary 7.3]{BF2}.

\begin{lem}\label{cor:hyperest}
There is an $x_0>2$ such that for $r \rightarrow \infty$ uniformly for  all $x>x_0$ we have
\begin{multline}\label{LG}
F\left( \frac{1}{4}+ir,\frac{3}{4}+ir,1+2ir; \frac{4}{x^2}\right)=x^{2ir}\exp(-2ir\arcosh{x/2}) \\\times
\left(
\frac{x^2}{x^2-4}\right)^{1/4}
\left(1+\frac{1}{16ir }\left( 1-\frac{x^2-2}{x\sqrt{x^2-4}}\right) \right)+O\left(\frac{1}{x^2r^2}
\right).
\end{multline}
\end{lem}

\begin{lem}\label{lemma2 I est z>2}
For $x>2$ the following estimate holds
\begin{multline}\label{Igeq2 estimate2}
\frac{1}{G\pi^{1/2}}\int_T^{2T}I(1/2;x)dK\ll
\left(\frac{x^2}{x^2-4}\right)^{1/4}
\exp(-G^2\arcosh^2{x/2})\min\left(T^{3/2},\frac{T^{1/2}}{\arcosh{x/2}}\right)\\\times
\left( 1+\frac{x}{T\sqrt{x^2-4}}\right)+\frac{1}{x^2T^{1/2}}.
\end{multline}
\end{lem}
\begin{proof}
As a consequence of \eqref{integralIgeq2}, we find that
\begin{multline}\label{integralIgeq2 1/2}
I(1/2;x)=\frac{2^{3/2}i}{\pi^{3/2}}\int_{-\infty}^{\infty}\frac{rh(r)}{\cosh(\pi r)}\left(\frac{2}{x}\right)^{2ir}
\sin\left( \pi(1/4-ir)\right)\frac{\Gamma(1/4+ir)\Gamma(3/4+ir)}{\Gamma(1+2ir)}\times \\ {}_2F_{1}\left(1/4+ir,3/4+ir,1+2ir;\frac{4}{x^2} \right)dr.
\end{multline}
Using the Stirling formula \eqref{Stirling2}, we expand the Gamma factors in the series with a negligible error term, such that the first term is equal to
\begin{equation*}
\frac{\Gamma(1/4+ir)\Gamma(3/4+ir)}{\Gamma(1+2ir)}=c|r|^{-1/2}2^{-2ir}+\ldots,
\end{equation*}
where $c$ is some absolute constant. Substituting \eqref{LG} to \eqref{integralIgeq2 1/2}, we have
\begin{multline}\label{LG appl1}
\frac{1}{G\pi^{1/2}}\int_T^{2T}I(1/2;x)dK\ll
\left(\frac{x^2}{x^2-4}\right)^{1/4}
\frac{1}{G\pi^{1/2}}\int_T^{2T}
\int_{-\infty}^{\infty}|r|^{1/2}h(r) \\\times
\exp(-2ir\arcosh{x/2})
\left(1+\frac{1}{16ir }\left( 1-\frac{x^2-2}{x\sqrt{x^2-4}}\right) \right)drdK+
O\left(\frac{1}{x^2T^{1/2}}\right).
\end{multline}
Using \eqref{hN def}, making the change of variables $r=K+Gy$ and applying \eqref{exp integral1}, we show that
\begin{multline*}
\frac{1}{G\pi^{1/2}}\int_T^{2T}
\int_{-\infty}^{\infty}|r|^{1/2}h(r)\exp(-2ir\arcosh{x/2})drdK\ll\\
\int_T^{2T}K^{1/2}\exp(-2iK\arcosh{x/2})\int_{-\infty}^{\infty}\exp(-y^2-2iGy\arcosh{x/2})dydK\\\ll
\exp(-G^2\arcosh^2{x/2})\int_T^{2T}K^{1/2}\exp(-2iK\arcosh{x/2})dK\\\ll
\exp(-G^2\arcosh^2{x/2})\min\left(T^{3/2},\frac{T^{1/2}}{\arcosh{x/2}}\right).
\end{multline*}
Estimating another integral in \eqref{LG appl1} in the same way, we complete the proof.
\end{proof}
\begin{cor}\label{cor I est z>2 rho=1/2}
The following estimate holds
\begin{equation}\label{sumIgeq2 estimate3}
\sum_{n=2l+1}^{\infty}\frac{1}{n^{1/2}}\mathscr{L}_{n^2-4l^2}(1/2)
\frac{1}{G\pi^{1/2}}\int_T^{2T}I\left(\frac{1}{2}; \frac{n}{l}\right)dK\ll \frac{T^{1/2+\epsilon}l^{1/2+2\theta}}{G^{1/2+2\theta}}.
\end{equation}
\end{cor}
\begin{proof}
Using \eqref{Igeq2 estimate0} and \eqref{Igeq2 estimate1},  we  show that the contribution of the sum over $n>2l+lT^{\epsilon}/G^2$ is negligible. To handle the remaining sum over $2l<n<2l+lT^{\epsilon}/G^2$, we apply \eqref{Igeq2 estimate2} and \eqref{eq:subconvexity}. As a result, we obtain that the left-hand side of \eqref{sumIgeq2 estimate3} is bounded by
\begin{multline}\label{sumIgeq2 estimate4}
\sum_{2l<n<2l+lT^{\epsilon}/G^2}\frac{(n^2-4l^2)^{\theta}}{n^{1/2}}
\Biggl(\left(\frac{n^2}{n^2-4l^2}\right)^{1/4}
\min\left(T^{3/2},\frac{T^{1/2}}{\arcosh{n/(2l)}}\right)\\\times
\left( 1+\frac{nl}{T\sqrt{n^2-4l^2}}\right)+\frac{l^2}{n^2T^{1/2}}\Biggr).
\end{multline}
We consider only the contribution of the summand with $1$. All other summands can be treated in the same way and are smaller in size.
This way we see that \eqref{sumIgeq2 estimate4} can be estimated as follows
\begin{multline}\label{sumIgeq2 estimate5}
T^{3/2}\sum_{2l<n<2l+l/T^2}(n^2-4l^2)^{\theta-1/4}+
\sum_{2l+l/T^2<n<2l+lT^{\epsilon}/G^2}(n^2-4l^2)^{\theta-1/4}\frac{T^{1/2}}{\arcosh{(n/(2l))}}\\\ll
\frac{l^{1/2+2\theta}}{T^{2\theta}}+
\sum_{2l+l/T^2<n<2l+lT^{\epsilon}/G^2}(n^2-4l^2)^{\theta-1/4}\frac{T^{1/2}l^{1/2}}{(n-2l)^{1/2}}\ll
\frac{T^{1/2+\epsilon}l^{1/2+2\theta}}{G^{1/2+2\theta}}.
\end{multline}
This completes the proof.
\end{proof}
\section*{Acknowledgement}
The reported study was funded by RFBR, project number 19-31-60029.


\end{document}